\RequirePackage{fix-cm}
\documentclass[smallextended]{svjour3}       
\smartqed  
\usepackage{graphicx}
\usepackage{amsmath,amssymb,amsfonts}
\newtheorem{thm}{Theorem}
\newtheorem{lem}{Lemma}
\newtheorem{prop}{Proposition}
\newtheorem{defn}{Definition}

\begin{document}

\title{Partial linearization for nonautonomous differential equations
}


\author{Xia Pan        \and
       Zuohuan Zheng 
}


\institute{Xia Pan \at
              Institute of Applied Mathematics, Academy of Mathematics and System Sciences, \\ Chinese Academy of
Sciences, Beijing 100190, China \\
              \email{ panxia13@mails.ucas.ac.cn}           
           \and
           Zuohuan Zheng \at
              Institute of Applied Mathematics, Academy of Mathematics and System Sciences, \\ Chinese Academy of
Sciences, Beijing 100190, China
}

\date{Received: date / Accepted: date}

\maketitle

\begin{abstract}
In this paper, we prove the partial linearization for $n$-dimensional nonautonomous differential equations. The conditions are formulated in terms of the dichotomy spectrum.
\keywords{Partial linearization \and $n$-dimensional nonautonomous differential equatuions \and Dichotomy spectrum}
\subclass{34A30 \and 34A34 \and 34D09}
\end{abstract}

\section{Introduction}
\label{intro}
Linearization is approximation or narrowing the scope under certain conditions, then we can deal with the nonlinear differential equation approximately as a linear differential equation. Usually, we expand the nonlinear form into Taylor series and eliminate higher order term, then we can obtain the approximate linear function.\\
 In 1879, Poincar\'{e} published his first major article concerning a class of automorphic functions. He found the normal form theory for autonomous differential equations $\dot{x}=f(x)$ near a rest point, which was a critical tool for Arnold's spectral sequence\cite{article1}. Spectral sequence was introduced by Vladimir Arnold in 1975. In 1959, Hartman and Grobman proved the Hartman-Grobman theorem or linearization theorem, which was a theorem about the dynamical systems in the neighbourhood of a hyperbolic equilibrium point\cite{article2}. The theorem states that the dynamics in a domain near a hyperbolic equilibrium point is equivalent to that of its linearization near this equilibrium point, provided that no eigenvalue of the linearization has its real part equal to zero. Therefore, when dealing with such dynamical systems one can use the simpler linearization of the system to analyze its behaviour around equilibria.\\
 In 2002, S.Siegmund\cite{article7} investigated lineariation of the nonautonomous differential equations. He extended Poincar\'{e}'s normal form theory to differential equations of $C^{k}$ Carath\'{e}odory type, $k\geq 2$, showing the system
 $$\dot{x}=f(t,x)$$
 is locally $C^{k}$ equivalent to a system
 $$\dot{x}=A(t)x+g(t,x)$$
  with zero reference solution, provided that the linearization along the reference solution satisfies a nonresonance condition. In 2015, P. Bonckaert, P. De Maesschalck, T. S. Doan and S.Siegmund \cite{article9} proved partial linearization for planar nonautonomous differential equations, they proved that the system
  $$\dot{x}=f(t,x), x\in\mathbb{R}^{2}, $$
  is equivalent to the following system
  $$\dot{x}=\begin{pmatrix}
\alpha_1(t) &         \\
            &\alpha_2(t)
\end{pmatrix}x+\begin{pmatrix}
p(t,x_{1})\\
q(t,x_{1})x_{2}
\end{pmatrix}.$$
However the result is only suitable for 2-dimensional nonautonomous differential equations. There are still no work on partial linearization for $n$-dimensional nonautonomous differential equations, $n\geq3$.
 We prove partial linearization for $n$-dimensional nonautonomous differential equations. We adopt  $C^{k}$ equivalence simplifying the equation step by step. To build a equivalence, we adopt the function of the form like E.Nelson's\cite{article3}. Then, under certain conditions, we can see that the original system is equivalent to a partial linearization system.\\
Exponential dichotomy introduced by Perron\cite{article13} plays an important role in the study of invariant manifolds. There are many paper\cite{article14}\cite{article15}\cite{article16}\cite{article18} concerning exponential dichotomy. Based on the study of classical exponential dichotomy, the dichotomy spectral theory was introduced by Sacker and Sell in \cite{article19}. The dichotomy spectrum is quite important in the field of dynamical systems, especially concerning nonautonomous differential equations, we can use dichotomy spectrum to get the normal forms\cite{article7}\cite{article17}. In this paper, we will use dichotomy spectrum as well.\\
 The structure of this paper is as follows. In section 2, we list three results which are key to our main result. The fist one is dichotomy spectrum for $n$-dimensional nonautonomous differential equations. The second one is normal forms for $n$-dimensional nonautonomous differential equations. The third one is very important. We can get a simplified system through the third one. In section 3, we show the existence of invariant manifold. In section 4, we estimate the higher derivatives of solutions. In section 5, we show the main theorem and prove it. In section 6, we give an example.

\section{Preliminaries}
Throughout this paper we use the following relations and marks:\\
1)$\lambda_{i}\triangleq[\underline{\lambda}_{i},\bar{\lambda}_{i}]$,$\lambda_{j}\triangleq[\underline{\lambda}_{j},\bar{\lambda}_{j}]$ ;\\
2)$\lambda_{i}+\lambda_{j}\triangleq[\underline{\lambda}_{i}+\underline{\lambda}_{j},\bar{\lambda}_{i}+\bar{\lambda}_{j}]$;\\
3)$\alpha\lambda_{i}\triangleq[\alpha\underline{\lambda}_{i},\alpha\bar{\lambda}_{i}]$, $\alpha\in\mathbb{R}$;\\
4)$\lambda_{i}<\lambda_{j}$ $\Longleftrightarrow$ $\bar{\lambda}_{i}<\underline{\lambda}_{j}$;\\
5)$B(t)\triangleq\begin{pmatrix}
\alpha_1(t) &         &        &\\
         &\alpha_2(t) &        &\\
         &         &\ddots  &\\
         &         &        &\alpha_n(t)
\end{pmatrix}$;\\

\subsection{\textit{Dichotomy spectrum for nonautonomous differential equations}}

In this section, we will show a spectral theorem which is proved by Siegmund\cite{article6}.\\
Consider a linear system of differential equations
\begin{equation}
  \dot{x}=A(t)x,
\end{equation}
with $A\in L_{loc}^{1}(\mathbb{R},\mathbb{R}^{N\times N})$, i.e., $A$ is a locally integrable matrix function. Let
$$\Phi:\mathbb{R}\times\mathbb{R}\rightarrow\mathbb{R}^{N\times N}:\  (t,\tau)\mapsto\Phi(t,\tau),$$
denote its \emph{evolution operator}, i.e., $\Phi(\cdot,\tau)\xi$ solves the initial value problem (1), with $x(\tau)=\xi$,~$\tau\in\mathbb{R},~\xi\in\mathbb{R}^{N}$.\\
An \emph{invariant projector} of (1) is defined to be a function $P:\mathbb{R}\rightarrow\mathbb{R}^{N\times N}$ of projections $P(t)$, $t\in\mathbb{R}$, such that
   $$P(t)\Phi(t,s)=\Phi(t,s)P(s)   \quad  \text{for} \quad t,s\in\mathbb{R}.$$
We shall say that (1) admits an \emph{exponential dichotomy} (\textbf{ED}) if there are an invariant projector $P$ and constants $K\geq1$ and $\alpha>0$ such that
  $$\parallel\Phi(t,s)P(s)\parallel\leq Ke^{-\alpha(t-s)} \quad\quad\text{for} \quad t\geq s,$$
  $$\parallel\Phi(t,s)[I-P(s)]\parallel\leq Ke^{\alpha(t-s)} \quad  \text{for} \quad t\leq s.$$
The \emph{dichotomy spectrum} of (1) is the set
     $$\Sigma(A)=\{\gamma\in\mathbb{R}: \dot{x}=[A(t)-\gamma I]x ~~\text{admits no ED}\}.$$
The structure of $\Sigma(A)$ is described in the following theorem:

\begin{thm}\cite{article6}
\label{th1}
  Suppose that (1) has \emph{bounded growth}, i.e., there exist constants $K\geq 1$ and $a\geq 0$ such that $$\parallel\Phi(t,s)\parallel\leq Ke^{a\mid t-s\mid}~~\text{for}~~ t,s\in \mathbb{R}.$$
  Then, the linear system (1) has a nonempty and compact dichotomy spectrum $$\Sigma(A)=[\underline{\lambda}_{1},\bar{\lambda}_{1}]\cup\cdots\cup[\underline{\lambda}_{n},\bar{\lambda}_{n}] ~~\text{where}~~ 1\leq n\leq N,$$
  and the spectral manifolds $\mathcal{W}_{0}$ and $\mathcal{W}_{n+1}$ are trivial, Moreover
  $$\mathcal{W}_{0}\oplus\cdots\oplus\mathcal{W}_{n+1}=\mathbb{R}\times\mathbb{R}^{N}.$$
 with the properties that for any $\varepsilon>0$ there exists $K(\varepsilon)\geq1$ such that the following statements hold:\\
  $$\parallel\Phi(t,s)\parallel\leq K(\varepsilon)e^{(\bar{\lambda}_{1}+\varepsilon)(t-s)}~~ \text{for}~~ t\geq s,$$
  $$\parallel\Phi(t,s)\parallel\leq K(\varepsilon)e^{(\underline{\lambda}_{n}-\varepsilon)(t-s)}~~\text{for}~~ t\leq s.$$
  \end{thm}

\subsection{\textit{Normal forms for nonautonomous differential equations}}
  We will introduce the local equivalence transformation for nonautonomous differential equations with respect to a fixed and given solution.
\begin{defn}\cite{article7}
  Let $N,M\in \mathbb{N}=\{1,2,\ldots\},D\subset\mathbb{R}\times\mathbb{R}^{N}$ open and $f:D\rightarrow \mathbb{R}^{M}$ a function.\\
  \begin{enumerate}
  \item Then $f$ is a \textit{Carath\'{e}odory function} if for every interval $I\subset \mathbb{R}$ and every open set $U\subset \mathbb{R}^{N}$ with $I\times U\subset D$ the following holds:
  \begin{enumerate}
   \item for a.a. $t\in I$ the mapping $$f\mid_{I\times U}(t,\cdot):U\rightarrow \mathbb{R}^{M},$$ is continuous,\\
   \item for all $x\in U$ the mapping
   $$f\mid_{I\times U}(\cdot,x):I\rightarrow \mathbb{R}^{M},$$
    is measurable (with respect to the Borel $\sigma$-algebras on $I$ and $\mathbb{R}^{M}$).
       \end{enumerate}
  \item Then $f$ is a \textit{$C^{k}$ Carath\'{e}odory function}, $k\geq 0$, if\\
   \begin{enumerate}
  \item  for a.a. $t\in \mathbb{R}$ and all $x\in \mathbb{R}^{N}$ with $(t,x)\in D$ the $\textit{k}$th partial derivative $D_{x}^{k}f(t,x)$ exists,\\
  \item  for every $j\in {0,\ldots,k}$ the mapping $$D_{x}^{k}f:D\rightarrow L^{j}(\mathbb{R}^{N},\mathbb{R}^{M}),$$   is a \textit{Carath\'{e}odory function}
       \end{enumerate}
   \end{enumerate}
\end{defn}

  Let us define $\varepsilon>0,$ $x_{0}\in \mathbb{R}^{N},$ $\mu:\mathbb{R}\rightarrow \mathbb{R}^{N}$ and the neighbourhoods
  $$B_{\varepsilon}(x_{0})=\{x\in \mathbb{R}^{N}:\|x-x_{0}\|<\varepsilon\},$$
  $$U_{\varepsilon}(\mu)=\{(t,x)\in \mathbb{R}\times\mathbb{R}^{N}:x\in B_{\varepsilon}(\mu(t))\}.$$
  Let us consider differential equations of $C^{k}$ Carath\'{e}odory type, $k\geq 0,$ with reference solutions
  \begin{equation}
     \dot{x}=f(t,x),~~\mu_{0}:\mathbb{R}\rightarrow \mathbb{R}^{N},
  \end{equation}
  \begin{equation}
     \dot{x}=g(t,x),~~\nu_{0}:\mathbb{R}\rightarrow \mathbb{R}^{N},
  \end{equation}
  i.e.,$f:D_{f}\subset \mathbb{R}\times\mathbb{R}^{N}\rightarrow\mathbb{R}^{N}$ and $g:D_{g}\subset \mathbb{R}\times\mathbb{R}^{N}\rightarrow\mathbb{R}^{N}$ are $C^{k}$ Carath\'{e}odory functions. We assume
  that tubular neighbourhood of the reference solution are contained in the corresponding sets of definition;
  i.e., there exist $r>0$ and $p>0$ such that
  $$U_{r}(\mu_{0})\subset D_{f}~~\text{and}~~U_{p}(\nu_{0})\subset D_{g}.$$

\begin{defn}\cite{article7}
  Consider the two equations (2) and (3). If there exist $r^{'}$, $p^{'}$ with $0<r^{'}\leq r$ and $0<p^{'}\leq p$
  together with continuous functions
  $$H:U_{r^{'}}(\mu_{0})\rightarrow \mathbb{R}^{N},~~H^{-1}:U_{p^{'}}(\nu_{0})\rightarrow \mathbb{R}^{N},$$
  then $H$ is called a local $C^{k}$ equivalence between the system (2) with solution $\mu_{0}$ and system (3) with solution $\nu_{0}$, if the following statements are valid:\\
 \begin{enumerate}
 \item  For each $t\in \mathbb{R}$ the mappings
      $$H(t,\cdot):B_{r^{'}}(\mu_{0}(t))\rightarrow H(t,B_{r^{'}}(\mu_{0}(t)))\subset B_{p}(\nu_{0}(t)), $$
      $$H^{-1}(t,\cdot):B_{p^{'}}(\nu_{0}(t))\rightarrow
      H^{-1}(t,B_{p^{'}}(\nu_{0}(t)))\subset B_{r}(\mu_{0}(t)), $$
      are $C^{k}$ diffeomorphisms (or homeomorphisms if $k=0$) with
      $$H(t,H^{-1}(t,x))=x,~~\text{for}~~x\in B_{p^{'}}(\nu_{0}(t)),$$
      $$H^{-1}(t,H(t,x))=x,~~\text{for}~~x\in B_{r^{'}}(\mu_{0}(t)).$$
  \item If $\mu$ is a solution of (2) in $U_{r^{'}}(\mu_{0})$ then$ H(\cdot,\mu(\cdot))$ is a solution of (3).\\
      If $\nu$ is a solution of (3) in $U_{p^{'}}(\nu_{0})$ then $H^{-1}(\cdot,\nu(\cdot))$ is a solution of (2).\\
  \item The reference solutions are mapped uniformly onto each other:
      $$\lim_{x\rightarrow 0}H(t,\mu_{0}(t)+x)=\nu_{0}(t)~~\text{uniformly in} ~~t\in\mathbb{R},$$

      $$\lim_{x\rightarrow 0}H^{-1}(t,\nu_{0}(t)+x)=\mu_{0}(t)~~\text{uniformly in} ~~t\in\mathbb{R}.$$
 \end{enumerate}
\end{defn}

\begin{thm}\cite{article7}
\label{th2}
  Consider a differential equation
  \begin{equation}
     \dot{x}=f(t,x),
  \end{equation}
   together with a reference solution $\mu_{0}:\mathbb{R}
  \rightarrow \mathbb{R}^{N}.$ Assume that\\
  \begin{enumerate}
  \item $f:D \subset \mathbb{R}\times \mathbb{R}^{N}\rightarrow \mathbb{R}^{N} is ~a~ C^{k}$
      Carath\'{e}odory function for $k\geq 2,$\\
  \item a neighbourhood $U_{r}(\mu_{0})$ is contained in $D$ for some $r>0,$\\
  \item the linearization $\dot{x}=D_{x}f(t,\mu_{0}(t))x$ of (4) along $\mu_{0}$ has bounded growth and therefore
      the dichotomy spectrum consists of $n$, $1\leq n\leq N$, compact intervals $\lambda_{i}=[\underline{\lambda}_{i},\bar{\lambda}_{i}], i=1,\ldots,n,$\\
 \item higher order terms of f in $x$ along $\mu_{0}$ are uniformly bounded in $t$, i.e., there is a $M>0$ such that\\
      $$\|D_{x}^{j}f(t,\mu_{0}(t))\|\leq M ~~\text{for a.a.}~~ t\in\mathbb{R}~~\text{and all}~~ j\in \{2,\ldots,k\}.$$

  \end{enumerate}
  Then (4) is locally $C^{k}$ equivalent to a differential equation
  \begin{equation}
    \dot{x}=g(t,x),
  \end{equation}
  with zero solution and (5) is in normal form, i.e., it holds that\\
  \begin{enumerate}
  \item g: $U_{q}(0)\rightarrow \mathbb{R}^{N}$ is a $C^{k}$ Carath\'{e}odory function for some $q>0,$\\
  \item the linearization $\dot{x}=D_{x}g(t,0)x$ of (5) along the zero solution has the same dichotomy spectrum
            as the linearization of (4) along $\mu_{0}$ and additionally is block-diagonalized, each block
            corresponds to a spectral interval $\lambda_{i},$\\
  \item all nontrivial Taylor components of g of order 2 to k are resonant, i.e., for every $j\in \{1,\ldots,n\}$
            and $l\in N_{o}^{n}, 2\leq |l|\leq k$ with
            $$\lambda_{j}\bigcap\sum_{i=1}^{n}l_{i}\lambda_{i}=\emptyset,$$
            we have $D_{x}^{l}g_{j}(t,0)\equiv0$ on $\mathbb{R}.$

  \end{enumerate}
\end{thm}
Now it is easy to get the following theorem.
\begin{thm}\cite{article9}
\label{th3}
 Consider the following nonautonomous planar differential equation
 \begin{equation}
   \dot{x}=f(t,x),
 \end{equation}
 with a given solution $\mu:\mathbb{R}\rightarrow\mathbb{R}^{n}$ which we assume to satisfy the following conditions:\\
 \begin{enumerate}
\item $f:D\subset\mathbb{R}\times\mathbb{R}^{n}\rightarrow\mathbb{R}^{n}$ is a  $C^{k}$ Carath\'{e}odory function, where D contains a neighborhood $B_{r}(\mu)$ for some $r>0$.\\
 \item The linearization of (6) along $\mu$ has bounded growth and its dichotomy spectrum $\Sigma_{dich}$ consists of n disjoint compact intervals $[\underline{\lambda}_{n},\bar{\lambda}_{n}]\cup\cdots\cup[\underline{\lambda}_{1},\bar{\lambda}_{1}]$, where $[\underline{\lambda}_{i},\bar{\lambda}_{i}]=\lambda_{i}$,~~$\lambda_{n}<\cdots<\lambda_{1}$.\\
 \item Higher order derivative terms of $f$ in $x$ are bounded in $B_{r}(\mu)$, i.e., for all $\alpha_{1},\cdots,\alpha_{n}\in\mathbb{N}$ there is an $M_{\alpha_{1},\cdots,\alpha_{n}}>0$ such that
 $$\parallel\frac{\partial^{\alpha_{1}+\cdots+\alpha_{n}}f_{1}}{\partial^{\alpha_{1}}x_{1}\cdots\partial^{\alpha_{n}}x_{n}}\parallel,\cdots,\parallel\frac{\partial^{\alpha_{1}+\cdots+\alpha_{n}}f_{1}}{\partial^{\alpha_{1}}x_{1}\cdots\partial^{\alpha_{n}}x_{n}}\parallel\leq M_{\alpha_{1},\cdots,\alpha_{n}}
 ~~\text{for}~~(t,x)\in B_{r}(\mu)$$
 where the norm is operate norm.
 \end{enumerate}
 The system (6) together with reference solution $\mu$ is locally $C^{k}$ equivalent to the following system
\begin{equation}
 \dot{x}=B(t)x+r_{1}(t,x),
\end{equation}
with the zero solution, where $r_{1}:\mathbb{R}\times\mathbb{R}^{n}\rightarrow\mathbb{R}^{n}$ is a  $C^{k}$ Carath\'{e}odory function satisfying the following properties:\\

\begin{enumerate}
\item $r_{1}(t,0)=0$ and $D_{x}r_{1}(t,0)=0$.\\
\item The partial derivatives of $r_{1}$ up to any finite order are globally bounded.\\
\item For $i=1,2,\cdots,n$, the dichotomy spectrum of the scalar equation $\dot{x_{i}}=a_{i}(t)x_{i}$ is $[\underline{\lambda}_{i},\bar{\lambda}_{i}]$, i.e., for all $\varepsilon>0$ there exists $K(\varepsilon)>0$ such that
$$\frac{1}{K(\varepsilon)}exp((\underline{\lambda}_{i}-\varepsilon)(t-s))\leq\Lambda_{i}(t,s)\leq K(\varepsilon)exp((\bar{\lambda}_{i}+\varepsilon)(t-s))$$
where $\Lambda_{i}(t,s)$ denotes the evolution operator of the equation $\dot{x_{i}}=a_{i}(t)x_{i}$.
\end{enumerate}
\end{thm}

\section{The existence of invariant manifold}
Let $J\subset(-\infty,\infty)$ be an interval. Let
$$C_{\rho}(J,\mathbb{R}^{n})=\{f:J\rightarrow\mathbb{R}^{n}~is~continuous~and~sup_{t\in J} e^{-\rho t}\parallel f(t)\parallel<\infty\},$$
for $\rho\in\mathbb{R}$.\\
In Zhang's paper \cite{article12}, which is detailed and well thought out on the existence of invariant manifold for nonautonomous differential equations. The following theorem is similar to Zhang's result.
\begin{thm} 
Consider system (7) satisfying the properties in the theorem \ref{th3}. For arbitrarily given constant $\rho\in(-\alpha, \alpha)$ satisfying that
$$ K(\epsilon)(\frac{1}{\alpha-\rho}+\frac{1}{\alpha+\rho})sup_{(t,x)\in\mathbb{R}\times\mathbb{R}^{n}}\parallel\frac{\partial r_{1}}{\partial x}(t,x)\parallel<1,$$
then the system (7) has global invariant manifolds
$\mathcal{M}_{\rho}^{+}=\{(\tau,x_{0})\in\mathbb{R}\times\mathbb{R}^{n}|x(t;\tau,x_{0}),~a~solution~of~(7)~passing~x_{0}~at~\tau,~is~in~C_{\rho}((-\infty,\tau),\mathbb{R}^{n})\}$ and            $\mathcal{M}_{\rho}^{-}=\{(\tau,x_{0})\in\mathbb{R}\times\mathbb{R}^{n}|x(t;\tau,x_{0}),~a~solution~of~(7)~passing~x_{0}~at~\tau,~is~in~C_{\rho}((\tau,+\infty),\mathbb{R}^{n})\}$. What's more
$$\mathcal{M}_{\rho}^{+}=\{(\tau,x_{0})\in\mathbb{R}\times\mathbb{R}^{n}|x_{0}=\xi+h(\tau,\xi),\xi\in\mathcal{R}(I-P(\tau))\},$$ and $$\mathcal{M}_{\rho}^{-}=\{(\tau,x_{0})\in\mathbb{R}\times\mathbb{R}^{n}|x_{0}=\xi+g(\tau,\xi),\xi\in\mathcal{R}P(\tau)\},$$ where $\mathcal{R}P(\tau)$ denotes the range of $P(\tau)$, $h(\tau,\cdot):\mathcal{R}(I-P(\tau))\rightarrow\mathcal{R}P(\tau)$ and $g(\tau,\cdot):\mathcal{R}P(\tau)\rightarrow\mathcal{R}(I-P(\tau))$ are continuous in $\tau$, and $h(\tau,0)\equiv0,g(\tau,0)\equiv0$.\\
\end{thm}

\begin{proof}
 For $\dot{x}=B(t)x$ has $\textbf{ED}$ and the partial derivatives of $r_{1}$ up to any finite order are globally bounded. we can see that our conditions satisfy the condition of theorem 3.1 in \cite{article12}, So the result is hold.
\end{proof}

According to above theorems, we start considering the following system
\begin{equation}
\dot{x}=B(t)x+r_{2}(t,x)
\end{equation}
with the zero solution, where $r_{2}:\mathbb{R}\times\mathbb{R}^{n}\rightarrow\mathbb{R}^{n}$ is a $C^{k}$ Carath\'{e}odory function satisfying the following properties:\\
\begin{enumerate}
\item $r_{2}(t,0)=0$ and $D_{x}r_{2}(t,0)=0$.\\
\item The partial derivatives of $r_{2}$ up to $k$ are globally bounded.\\
\item The set $\underbrace{\mathbb{R}\times\cdots\times\mathbb{R}}_{n-1}\times\{0\}$,$\underbrace{\mathbb{R}\times\cdots\times\mathbb{R}}_{n-2}\times\{0\}\times\{0\}$,$\cdots$,$\mathbb{R}\times\underbrace{\{0\}\times\cdots\times\{0\}}_{n-1}$
 are invariant manifold.
\end{enumerate}

\section{Estimate on higher derivatives of solutions}
 We have proved that under a suitable spectral gap condition, a nonautonomous differential equation together with a reference solution can be transformed by $C^{k}$ equivalence into a nonautonomous differential equation with zero solution which satisfies that the linear part of the linearization is diagonal and the set $\underbrace{\mathbb{R}\times\cdots\times\mathbb{R}}_{n-1}\times\{0\}$,$\underbrace{\mathbb{R}\times\cdots\times\mathbb{R}}_{n-2}\times\{0\}\times\{0\}$,$\cdots$,$\mathbb{R}\times\underbrace{\{0\}\times\cdots\times\{0\}}_{n-1}$ are invariant manifold.\\
Let $p_{i}:\mathbb{R}\times\mathbb{R}^{n-1}\rightarrow\mathbb{R}$ be $C^{k}$ Carath\'{e}odory functions satisfies that all partial derivatives of $p_{i}$ with respect to $x$ are globally bounded, and for $t\in\mathbb{R}$
$$\| p_{i}(t,x_{1},\cdots,x_{n-1})\|\leq\delta,~~1\leq i \leq n;$$
$$\|\begin{pmatrix}
\frac{\partial p_{1}}{\partial x_{1}}&\frac{\partial p_{1}}{\partial x_{2}}&\cdots&\frac{\partial p_{1}}{\partial x_{n-1}}\\
\frac{\partial p_{2}}{\partial x_{1}}&\frac{\partial p_{2}}{\partial x_{2}}&\cdots&\frac{\partial p_{2}}{\partial x_{n-1}}\\
\cdots&\cdots&\cdots&\cdots\\
\frac{\partial p_{n-1}}{\partial x_{1}}&\frac{\partial p_{n-1}}{\partial x_{2}}&\cdots&\frac{\partial p_{n-1}}{\partial x_{n-1}}&
\end{pmatrix}\|\leq\frac{\delta}{2K(\varepsilon)}.$$
where $\varepsilon,\delta$ are arbitrary positive numbers and $K(\varepsilon)$ is chosen such that Theorem \ref{th1} holds. Additionally, in the case $\bar{\lambda}_{1}<0$ the positive constants $\varepsilon,\delta$ are assumed to satisfy $\bar{\lambda}_{1}+\varepsilon+\delta<0$.\\
Consider the following partially linear system

  \begin{equation}
    \left\{
    \begin{array}{l}
    \dot{x}_{1}=a_{1}(t)x_{1}+p_{1}(t,x_{1},\cdots,x_{n-1})\\
    ~~~~~~~~~~~\cdots\\
    \dot{x}_{n-1}=a_{n-1}(t)x_{n-1}+p_{n-1}(t,x_{1},\cdots,x_{n-1})\\
    \dot{x}_{n}=a_{n}(t)x_{n}+p_{n}(t,x_{n},\cdots,x_{n-1})x_{n}\\
    \end{array}
    \right.
  \end{equation}

Let $\phi(\cdot,s,x_{1},\cdots,x_{n})$ denote the unique solution of (9) which satisfies that $x(s)=(x_{1},\cdots,x_{n})^{T}$. In the following lemma, we state some properties of $\phi=(\phi_{1},\cdots,\phi_{n})$.\\

\begin{lem}
\label{lemma1}
The following statements hold£º\\
    \begin{enumerate}
   \item $\phi_{j}(t,s,x_{1},\cdots,x_{n}),1\leq j\leq n-1$ are independent of $x_{n}$.\\
    \item $\phi_{n}(t,s,x_{1},\cdots,x_{n})=\Lambda_{n}(t,s)x_{n}e^{\int_{s}^{t}p_{n}(u,\phi_{1}(u,s,x),\cdots,\phi_{n-1}(u,s,x))du}.$
    \end{enumerate}
\end{lem}
\begin{proof}
The first one is obvious. Let me prove the second one. we can change the following equation
 $$\dot{x}_{n}=a_{n}(t)x_{n}+p_{n}(t,x_{n},\cdots,x_{n-1})x_{n},$$
 into
 $$\frac{\dot{x}_{n}}{{x}_{n}}=a_{n}(t)+p_{n}(t,x_{n},\cdots,x_{n-1}).$$
 Because $\Lambda_{n}(t,s)=e^{\int_{s}^{t}a_{n}(u)du}$, we get
 $${x}_{n}(t)=\Lambda_{n}(t,s)x_{n}e^{\int_{s}^{t}p_{n}(u,\phi_{1}(u,s,x),\cdots,\phi_{n-1}(u,s,x)du},$$
 then completes the proof of Lemma \ref{lemma1}.
 \end{proof}
\begin{lem}
\label{lemma2}
There exists a positive constant $L$ such that for any $1\leq i_{1}+\cdots+i_{n-1}\leq k$ the following statements hold for all
$t\geq s$, $j=1,\cdots,n-1$ and $x\in \mathbb{R}^{n}$: \\
\begin{enumerate}
    \item If $\bar{\lambda}_{1}\geq 0$, we have $\|\frac{\partial^{i_{1}+\cdots+i_{n-1}} \phi_{j}}{\partial^{i_{1}} x_{1}\cdots\partial^{i_{n-1}}x_{n-1}}\|\leq Le^{(i_{1}+\cdots+i_{n-1})(\bar{\lambda}_{1}+\varepsilon+\delta)(t-s)},$  \\
   \item If $\bar{\lambda}_{1}<0$, we have $\|\frac{\partial^{i_{1}+\cdots+i_{n-1}} \phi_{j}}{\partial^{i_{1}} x_{1}\cdots\partial^{i_{n-1}}x_{n-1}}\|\leq Le^{(\bar{\lambda}_{1}+\varepsilon+\delta)(t-s)}$.  \\
    \end{enumerate}
\end{lem}
\begin{proof}
We prove by induction on $i_{1}+\cdots+i_{n}=m$ the following inequality for $x\in\mathbb{R}^{n}$
\[
\|\frac{\partial^{i_{1}+\cdots+i_{n}}\phi^{'}}{\partial^{i_{1}} x_{1}\cdots\partial^{i_{n}}x_{n}}\|\leq Le^{((i_{1}+\cdots+i_{n})(\bar{\lambda}_{1}+\varepsilon+\delta)-\frac{\delta}{2})(t-s)}~~\text{for}~~t\geq s.
\]
where $\phi^{'}$ is the solution of the following system
\[
\dot{x}=B(t)x+r(t,x),
\] with the initial condition $x(s)=x$ and $\|D_{x}r(u,\phi^{'}(u,s,x))\|\leq\frac{\delta}{2K(\varepsilon)}$.\\
Let $\Lambda(t,s)x$ denote the solution of the following system
\[
\dot{x}=B(t)x,
\] with the initial condition $x(s)=x$, we have
\[
\phi^{'}(t,s,x)=\Lambda(t,s)x+\int_{s}^{t}\Lambda(t,u)r(u,\phi^{'}(u,s,x))du,
\]\\
then
\[
D_{x}\phi^{'}(t,s,x)=\Lambda(t,s)+\int_{s}^{t}\Lambda(t,u)D_{x}r(u,\phi^{'}(u,s,x))D_{x}\phi^{'}(u,s,x)du.
\]
So
\[
\|D_{x}\phi^{'}(t,s,x)\|\leq K(\varepsilon)e^{(\bar{\lambda}_{1}+\varepsilon)(t-s)}+\frac{\delta}{2}\int_{s}^{t}e^{(\bar{\lambda}_{1}+\varepsilon)(t-u)}\|D_{x}\phi^{'}(u,s,x)\|du,
\]
by Gronwall's inequality yields
\[
\|e^{(\bar{\lambda}_{1}+\varepsilon)(s-t)}D_{x}\phi^{'}(t,s,x)\|\leq K(\varepsilon)e^{\frac{\delta}{2}(t-s)}~~for~~t\geq s.
\]
 so, the inequality holds for $m=1$. Suppose that the assertion holds for $m-1$.
 when $i_{1}+\cdots+i_{n}=m$,
 \[
 \frac{\partial^{m}r(u,\phi^{'}(u,s,x))}{\partial^{i_{1}} x_{1}\cdots\partial^{i_{n}}x_{n}}=D_{x}r(u,\phi^{'}(u,s,x))\frac{\partial^{m}\phi^{'}(u,s,x)}{\partial^{i_{1}} x_{1}\cdots\partial^{i_{n}}x_{n}}+R(u,s,x),
 \]
 where $\|R(u,s,x)\|\leq Ce^{(m(\bar{\lambda}_{1}+\varepsilon+\delta)-\delta)(u-s)}$, $C$ is a constant.\\
 So
 \[
 \frac{\partial^{m}\phi^{'}(t,s,x)}{\partial^{i_{1}} x_{1}\cdots\partial^{i_{n}}x_{n}}=\int_{s}^{t}\Lambda(t,u)(D_{x}r(u,\phi^{'}(u,s,x))\frac{\partial^{m}\phi^{'}(u,s,x)}{\partial^{i_{1}} x_{1}\cdots\partial^{i_{n}}x_{n}}+R(u,s,x))du.
 \]
 Consequently,
 \[
 \begin{split}
 \|\frac{\partial^{m}\phi^{'}(t,s,x)}{\partial^{i_{1}} x_{1}\cdots\partial^{i_{n}}x_{n}}\|\leq&\frac{\delta}{2}\int_{s}^{t}e^{(\bar{\lambda}_{1}+\varepsilon)(t-u)}\|\frac{\partial^{m}\phi^{'}(u,s,x)}{\partial^{i_{1}} x_{1}\cdots\partial^{i_{n}}x_{n}}\|du\\
 &+CK(\varepsilon)\int_{s}^{t}e^{(\bar{\lambda}_{1}+\varepsilon)(t-u)}e^{(m(\bar{\lambda}_{1}+\varepsilon+\delta)-\delta)(u-s)}du,
 \end{split}
 \]
 by Gronwall's inequality yields
 \[
 \|\frac{\partial^{m}\phi^{'}(t,s,x)}{\partial^{i_{1}} x_{1}\cdots\partial^{i_{n}}x_{n}}\|\leq Le^{(m(\bar{\lambda}_{1}+\varepsilon+\delta)-\frac{\delta}{2})(t-s)}.
 \]
 So the assertion is proved, Lemma \ref{lemma2} is a special case of the assertion. Now, we complete the prove of Lemma \ref{lemma2}.
\end{proof}

According to Lemmma\ref{lemma1} and Lemma\ref{lemma2}, we can get the following lemma:
\begin{lem}
 There exists a positive constant $L$ such that the following statements hold for all $t\geq s$ and $x\in \mathbb{R}^{n}$:\\
 \begin{enumerate}
 \item If $\bar{\lambda}_{1}>0$, we have \\
  $$\parallel\frac{\partial^{i} \phi_{n}}{\partial^{i_{1}} x_{1}\cdots\partial^{i_{n-1}}x_{n-1}}\parallel\leq Le^{((\bar{\lambda}_{n}+\varepsilon+\delta)+i(\bar{\lambda}_{1}+\varepsilon+\delta))(t-s)}|x_{n}|,$$
 for $0\leq i_{1}+\cdots+i_{n-1}=i\leq N$,
 $$\parallel\frac{\partial^{i+1} \phi_{n}}{\partial^{i_{1}} x_{1}\cdots\partial^{i_{n-1}}x_{n-1}\partial x_{n}}\parallel\leq Le^{((\bar{\lambda}_{n}+\varepsilon+\delta)+i(\bar{\lambda}_{1}+\varepsilon+\delta))(t-s)},$$
for $0\leq i_{1}+\cdots+i_{n-1}=i\leq N-1$.\\
\item If $\bar{\lambda}_{1}<0$, we have \\
 $$\parallel\frac{\partial^{i} \phi_{n}}{\partial^{i_{1}} x_{1}\cdots\partial^{i_{n-1}}x_{n-1}}\parallel\leq Le^{(\bar{\lambda}_{n}+\varepsilon+\delta)(t-s)}|x_{n}|~~for~~0\leq i_{1}+\cdots+i_{n-1}=i\leq N,$$
 $$\parallel\frac{\partial^{i+1} \phi_{n}}{\partial^{i_{1}} x_{1}\cdots\partial^{i_{n-1}}x_{n-1}\partial x_{n}}\parallel\leq Le^{(\bar{\lambda}_{n}+\varepsilon+\delta)(t-s)}~~for~~0\leq i_{1}+\cdots+i_{n-1}=i\leq N-1.$$
 \end{enumerate}
\end{lem}

\section{Partial linearization}
Denote
$$P(t,x)=\begin{pmatrix}
p_{1}(t,x_{1},\cdots,x_{n-1})\\
         \cdots\\
p_{n-1}(t,x_{1},\cdots,x_{n-1})\\
p_{n}(t,x_{1},\cdots,x_{n-1})x_{n}
\end{pmatrix},$$
$$Q(t,x)=\begin{pmatrix}
b_{1}(t,x_{1},\cdots,x_{n-1})x_{n}^{l}\\
         \cdots\\
b_{n-1}(t,x_{1},\cdots,x_{n-1})x_{n}^{l}\\
0
\end{pmatrix},$$
$$\Lambda(t,s)=\begin{pmatrix}
\Lambda_1(t,s) &               &        &\\
               &\Lambda_2(t,s) &        &\\
               &               &\ddots  &\\
               &               &        &\Lambda_n(t,s)
\end{pmatrix},$$

\subsection{\textit{Normalization up to finite order along the flat invariant manifold}}
Denote $C(m,n)$=\{ Carath\'{e}odory function $\alpha:R\times R^{n}\rightarrow R^{n}$ $\mid$ if there exist $C^{m-j}$ Carath\'{e}odory functions $\alpha_{j}(t,x_{1},\cdots,x_{n-1})$ for $j=0,\cdots,n-1$, such that the function
$\hat{\alpha}(t,x):=\alpha(t,x)-\sum_{j=0}^{n-1}\alpha_{j}(t,x_{1},\cdots,x_{n-1})x_{n}^{j}$\}

Of course any $C^{m}$ Carath\'{e}odory function $\alpha(t,x)$ is of class $C(m,n)$ with $n\leq m$ and the corresponding $\alpha_{j}$ are just derived from Taylor's theorem.\\
Then we have
\begin{equation}
\label{eq1}
\dot{x}=B(t)x+P(t,x)+Q(t,x)+r_{3}(t,x),
\end{equation}
where $b_{j}(t,x_{1},\cdots,x_{n-1})$, $1\leq j\leq n-1$ are $C^{m+1-l}$ Carath\'{e}odory functions and $r_{3}(t,x)$ belongs to $C(m+1,n)$ and $r_{3}(t,x)=O(x_{n}^{l+1})$.
\begin{equation}
\label{eq2}
\dot{x}=B(t)x+P(t,x)+r_{4}(t,x),
\end{equation}
where $r_{4}(t,x)$ belongs to $C(m,n)$ and $r_{4}(t,x)=O(x_{n}^{l+1})$.\\
\textbf{Condition 1}\\
a)$\lambda_{n}<\cdots<\lambda_{1}$, $\lambda_{n}<0$;\\
b)$l\bar{\lambda}_{n}-\underline{\lambda}_{n-1}+(m+1-l)\bar{\lambda}_{1}+(m+2l)(\varepsilon+\delta)<0$ as $ \bar{\lambda}_{1}\geq0$;\\
c)$l\bar{\lambda}_{n}-\underline{\lambda}_{n-1}+\bar{\lambda}_{1}+(m+2)(\varepsilon+\delta)<0$ as $ \bar{\lambda}_{1}<0$.\\
where $1\leq l\leq n-1$ and  $m\geq n$ are arbitrary, $\varepsilon$ and $\delta$ satisfied Theorem \ref{th3} ;
\begin{prop}
\label{p1}
If \textbf{Condition 1} hold, then system (\ref{eq1}) is $C^{m+1-l}$ equivalent to the system (\ref{eq2}).
\end{prop}

\begin{proof}
Let $\varphi(t,s,x_{1},\cdots,x_{n})$ denotes the solution of (10) with the initial condition $x(s)=(x_{1},\cdots,x_{n})$. Recall that $\Lambda(t,s)x$ denote the solution of the following system
\[
\dot{x}=B(t)x,
\] with the initial condition $x(s)=x$,
$\phi(t,s,x)$ denotes the solution of the following system
\[
\dot{x}=B(t)x+P(t,x).
\] with the initial condition $x(s)=x$.\\
In order to construct a smooth transformation between system (10) and (11), we define a scalar function $h_{s,t}:\mathbb{R}^{n-1}\rightarrow\mathbb{R}^{n-1}$ for $s,t\in \mathbb{R}$ by
$$h_{s,t}(x_{1},\cdots,x_{n-1}):=\begin{pmatrix}
\lim\limits_{x_{n}\rightarrow 0}\frac{\varphi_{1}(s,t)\phi(t.s)(x_{1},\cdots,x_{n})-x_{1}}{x_{n}^{l}}\\
\cdots\\
\lim\limits_{x_{n}\rightarrow 0}\frac{\varphi_{n-1}(s,t)\phi(t.s)(x_{1},\cdots,x_{n})-x_{n-1}}{x_{n}^{l}}
\end{pmatrix}~~for~~x_{1},\cdots,x_{n-1}\in \mathbb{R}.$$
In the following steps, we investigate some properties of the function $h_{s,t}$ and explain how to construct the desired transformation by using the function $h_{s,t}$.
Step 1:In this step, we prove that
$$h_{s,t}(x_{1},\cdots,x_{n-1})=\\-\int_{s}^{t}
\Lambda\cdot\Lambda_{n}^{l}g_{u,s}(x_{1},\cdots,x_{n-1})du,$$
where
\begin{equation}
\begin{split}
&g_{u,s}(x_{1},\cdots,x_{n-1})=exp(l\int_{s}^{u}p_{n}(v,\phi_{1},\cdots,\phi_{n-1})dv)\cdot\\
 &exp(-\int_{s}^{u}\begin{pmatrix}
\frac{\partial p_{1}}{\partial x_{1}}&\frac{\partial p_{1}}{\partial x_{2}}&\cdots&\frac{\partial p_{1}}{\partial x_{n-1}}\\
\frac{\partial p_{2}}{\partial x_{1}}&\frac{\partial p_{2}}{\partial x_{2}}&\cdots&\frac{\partial p_{2}}{\partial x_{n-1}}\\
\cdots&\cdots&\cdots&\cdots\\
\frac{\partial p_{n-1}}{\partial x_{1}}&\frac{\partial p_{n-1}}{\partial x_{2}}&\cdots&\frac{\partial p_{n-1}}{\partial x_{n-1}}&
\end{pmatrix}dv)\cdot\begin{pmatrix}
b_{1}\\
b_{2}\\
\vdots\\
b_{n-1}
\end{pmatrix}.
\end{split}
\end{equation}
For this purpose,for each $\tau,t,s,x_{1}\in \mathbb{R}$ we define
\begin{equation}
\begin{split}
&R(\tau,t,s,x_{1},\cdots,x_{n-1}):=\\
&\begin{pmatrix}
\lim\limits_{x_{n}\rightarrow 0}\frac{\varphi_{1}(\tau,t)\phi(t,s,x_{1},\cdots,x_{n})-\phi_{1}(\tau,s,x_{1},\cdots,x_{n-1})}{x_{n}^{l}}\\
\lim\limits_{x_{n}\rightarrow 0}\frac{\varphi_{2}(\tau,t)\phi(t,s,x_{1},\cdots,x_{n})-\phi_{2}(\tau,s,x_{1},\cdots,x_{n-1})}{x_{n}^{l}}\\
\cdots\\
\lim\limits_{x_{n}\rightarrow 0}\frac{\varphi_{n-1}(\tau,t)\phi(t,s,x_{1},\cdots,x_{n})-\phi_{n-1}(\tau,s,x_{1},\cdots,x_{n-1})}{x_{n}^{l}}
\end{pmatrix}.
\end{split}
\end{equation}
Since $\varphi$ is a solution of (10), it follows that
\begin{equation}
\begin{split}
&\frac{d}{d\tau}(\varphi_{j}(\tau,t)\phi(t,s)(x))\\
&=a_{j}(\tau)\varphi_{j}(\tau,t)\phi(t,s)(x)+\\
&p_{j}(\tau,\varphi_{1}(\tau,t)\phi(t,s)(x),\cdots,\varphi_{n-1}(\tau,t)\phi(t,s)(x))+\\
&b_{j}(\tau,\varphi_{1}(\tau,t)\phi(t,s)(x),\cdots,\varphi_{n-1}(\tau,t)\phi(t,s)(x))(\varphi_{n}(\tau,t)\phi(t,s)(x))^{l}+O(|x_{n}|^{l+1}).
\end{split}
\end{equation}
Since $\phi$ is a solution of (9), it follows that
\begin{equation}
\begin{split}
&\frac{d}{d\tau}(\phi_{j}(\tau,s,x_{1},\cdots,x_{n-1}))\\
&=a_{j}(\tau)\phi_{j}(\tau,s,x_{1},\cdots,x_{n-1})+\\
&p_{j}(\tau,\phi_{1}(\tau,s,x_{1},\cdots,x_{n-1}),\cdots,\phi_{n-1}(\tau,s,x_{1},\cdots,x_{n-1})).
\end{split}
\end{equation}

Using the Mean Value Theorem, we obtain
\begin{equation}
\begin{split}
&\lim\limits_{x_{n}\rightarrow 0}\frac{p_{j}(\tau,\varphi_{1}(\tau,t)\phi(t,s)(x),\cdots,\varphi_{n-1}(\tau,t)\phi(t,s)(x))-p_{j}(\tau,\phi_{1}(\tau,s,x),\cdots,\phi_{n-1}(\tau,s,x))}{x_{n}^{l}}
\\
&=\frac{\partial p_{j}}{\partial x_{1}}(\tau,\phi_{1}(\tau,s,x),\cdots,\phi_{n-1}(\tau,s,x)){x_{n}^{l}})\lim\limits_{x_{n}\rightarrow 0}\frac{\varphi_{1}(\tau,t)\phi(t,s,x)-\phi_{1}(\tau,s,x)}{x_{n}^{l}}+\\
&\cdots+\\
&\frac{\partial p_{j}}{\partial x_{n-1}}(\tau,\phi_{1}(\tau,s,x),\cdots,\phi_{n-1}(\tau,s,x)){x_{n}^{l}})\lim\limits_{x_{n}\rightarrow 0}\frac{\varphi_{n-1}(\tau,t)\phi(t,s,x)-\phi_{n-1}(\tau,s,x)}{x_{n}^{l}}.
\end{split}.
\end{equation}
Using $\lim\limits_{x_{n}\rightarrow 0}\varphi_{j}(\tau,t,\Phi(t,s,x)=\Phi_{j}(\tau,s,x_{1},\cdots,x_{n-1})$
then we have
\begin{equation}
\begin{split}
&\frac{d}{d\tau}R(\tau,t,s,x_{1},\cdots,x_{n-1})\\
&=\left[\begin{pmatrix}
\alpha_1 &         &        &\\
         &\alpha_2 &        &\\
         &         &\ddots  &\\
         &         &        &\alpha_n
\end{pmatrix}+\begin{pmatrix}
\frac{\partial p_{1}}{\partial x_{1}}&\frac{\partial p_{1}}{\partial x_{2}}&\cdots&\frac{\partial p_{1}}{\partial x_{n-1}}&\\
\frac{\partial p_{2}}{\partial x_{1}}&\frac{\partial p_{2}}{\partial x_{2}}&\cdots&\frac{\partial p_{2}}{\partial x_{n-1}}&\\
               \cdots&\cdots&\cdots&\cdots\\
\frac{\partial p_{n-1}}{\partial x_{1}}&\frac{\partial p_{n-1}}{\partial x_{2}}&\cdots&\frac{\partial p_{n-1}}{\partial x_{n-}}&
\end{pmatrix}\right]R\\
&+\begin{pmatrix}
b_{1}\\
b_{2}\\
\vdots\\
b_{n-1}
\end{pmatrix}\lim\limits_{x_{n}\rightarrow 0}(\frac{\varphi_{n}(\tau,t,\Phi(t,s,x))}{x_{n}})^{l},
\end{split}
\end{equation}
\begin{equation}
\begin{split}
&\varphi_{n}(\tau,t,\phi(t,s,x))\\
&=\Lambda_{n}(\tau,t)\phi_{n}(t,s,x)e^{\int_{t}^{\tau}p_{n}(u,\varphi_{1},\cdots,\varphi_{n-1})du}+O(x_{n}^{l+1})\\
&=\Lambda_{n}(\tau,t)\Lambda_{n}(t,s)x_{n}e^{\int_{s}^{t}p_{n}(u,\phi_{1},\cdots,\phi_{n-1})du}\\
&\cdot e^{\int_{t}^{\tau}p_{n}(u,\varphi_{1},\cdots,\varphi_{n-1})dv}+O(x_{n}^{l+1}).
\end{split}
\end{equation}

Consequently,
$$
\lim\limits_{x_{n}\rightarrow 0}\frac{\varphi_{n}(\tau,t,\Phi(t,s,x))}{x_{n}}=\Lambda_{n}(\tau,s)e^{\int_{s}^{\tau}p_{n}(u,\Phi_{1},\cdots,\Phi_{n-1})du}.
$$
Therefore,
\begin{equation}
\begin{split}
&\frac{d}{d\tau}R(\tau,t,s,x_{1},\cdots,x_{n-1})\\
&=\left[B(t)+\begin{pmatrix}
\frac{\partial p_{1}}{\partial x_{1}}&\frac{\partial p_{1}}{\partial x_{2}}&\cdots&\frac{\partial p_{1}}{\partial x_{n-1}}&\\
\frac{\partial p_{2}}{\partial x_{1}}&\frac{\partial p_{2}}{\partial x_{2}}&\cdots&\frac{\partial p_{2}}{\partial x_{n-1}}&\\
               \cdots&\cdots&\cdots&\cdots\\
\frac{\partial p_{n-1}}{\partial x_{1}}&\frac{\partial p_{n-1}}{\partial x_{2}}&\cdots&\frac{\partial p_{n-1}}{\partial x_{n-1}}&
\end{pmatrix}\right]R\\
&+\begin{pmatrix}
b_{1}\\
b_{2}\\
\vdots\\
b_{n-1}
\end{pmatrix}\Lambda_{n}(\tau,s)^{l}e^{l\int_{s}^{\tau}p_{n}(u,\Phi_{1},\cdots,\Phi_{n-1})du}.
\end{split}
\end{equation}

Then,
$$
R(\tau,t,s,x_{1},\cdots,x_{n-1})=\int_{t}^{\tau}\Lambda\cdot\Lambda_n(u,s)^{l}g_{u,\tau}(x_{1},\cdots,x_{n-1})du,
$$
where
\begin{equation}
\begin{split}
g_{u,\tau}(x_{1},\cdots,x_{n-1})=&exp(\int_{t}^{\tau}\begin{pmatrix}
\frac{\partial p_{1}}{\partial x_{1}}&\frac{\partial p_{1}}{\partial x_{2}}&\cdots&\frac{\partial p_{1}}{\partial x_{n-1}}&\\
\frac{\partial p_{2}}{\partial x_{1}}&\frac{\partial p_{2}}{\partial x_{2}}&\cdots&\frac{\partial p_{2}}{\partial x_{n-1}}&\\
               \cdots&\cdots&\cdots&\cdots\\
\frac{\partial p_{n-1}}{\partial x_{1}}&\frac{\partial p_{n-1}}{\partial x_{2}}&\cdots&\frac{\partial p_{n-1}}{\partial x_{n-1}}&
\end{pmatrix}dv)\\
&\cdot\begin{pmatrix}
b_{1}\\
b_{2}\\
\vdots\\
b_{n-1}
\end{pmatrix}exp(l\int_{s}^{u}p_{n}(v,\Phi_{1},\cdots,\Phi_{n-1})dv).
\end{split}
\end{equation}

Because $h_{s,t}(x_{1},\cdots,x_{n-1})=R(s,t,x_{1},\cdots,x_{n-1})$, we proved the equation.\\

Step 2: We will show the existence of $h_{s}:\mathbb{R}^{n-1}\rightarrow\mathbb{R}^{n-1}$, $h_{s}$ is defined as follow:
$$h_{s}(x_{1},\cdots,x_{n-1})=\lim\limits_{x_{n}\rightarrow 0}h_{s,t}(x_{1},\cdots,x_{n-1}),$$

let  $t_{2}>t_{1}>s$, then
\begin{equation}
\begin{split}
&\parallel h_{s,t_{1}}(x_{1},\cdots,x_{n-1})-h_{s,t_{2}}(x_{1},\cdots,x_{n-1})\parallel\\
&\leq\int_{t_{1}}^{t_{2}}\parallel\begin{pmatrix}
\Lambda_1 &         &        &\\
         &\Lambda_2 &        &\\
         &         &\ddots  &\\
         &         &        &\Lambda_{(n-1)}
\end{pmatrix}\parallel\cdot\\
&\parallel\Lambda_{n}(u,s)^{l}\parallel\cdot\parallel g_{u,s}(x_{1},\cdots,x_{n-1})\parallel du
\end{split}
\end{equation}
which implies that
\begin{equation}
\begin{split}
&\parallel h_{s,t_{1}}(x_{1},\cdots,x_{n-1})-h_{s,t_{2}}(x_{1},\cdots,x_{n-1})\parallel\\
&\leq K(\varepsilon)^{l+1}\int_{t_{1}}^{t_{2}}e^{(l(\bar{\lambda}_{n}+\varepsilon)-\underline{\lambda}_{n-1}+\varepsilon)}\parallel g_{u,s}(x_{1},\cdots,x_{n-1})\parallel du.
\end{split}
\end{equation}

By the definition of $g$ and the boundedness of the derivatives of $p_{i}, 1\leq i\leq n$, we get
$$\parallel g_{u,s}(x_{1},\cdots,x_{n-1})\parallel\leq M\cdot e^{(l+1)\delta(u-s)}~~for~~u\geq s.$$

Therefore
$$
\parallel h_{s,t_{1}}(x_{1},\cdots,x_{n-1})-h_{s,t_{2}}(x_{1},\cdots,x_{n-1})\parallel$$
$$\leq M\cdot K(\varepsilon)^{l+1}\int_{t_{1}}^{t_{2}}e^{(l\bar{\lambda}_{n}-\underline{\lambda}_{n-1}+(l+1)(\varepsilon+\delta))(u-s)}du.
$$
$l\bar{\lambda}_{n}-\underline{\lambda}_{n-1}+(l+1)(\varepsilon+\delta)<0$ implies that the limit $\lim\limits_{x_{n}\rightarrow 0}h_{s,t}(x_{1},\cdots,x_{n-1})$ exists uniformly in $x_{1},\cdots,x_{n-1}$.\\

Step 3: $h_{s}$ is $C^{m+1-l}$, that is to say, the following improper integral (23) is uniform convergence
\begin{equation}
\int_{s}^{\infty}\begin{pmatrix}
\Lambda_1 &         &        &\\
         &\Lambda_2 &        &\\
         &         &\ddots  &\\
         &         &        &\Lambda_{(n-1)}
\end{pmatrix}\cdot\Lambda_{n}(u,s)^{l}\cdot\frac{\partial^{i}g_{u,s}}{\partial x_{1}^{i_{1}}\cdots x_{n-1}^{i_{n-1}}}(x_{1},\cdots,x_{n-1})du,
\end{equation}
for $i=i_{1}+\cdots+i_{n-1}=1,\cdots,m+1-l$.

Case1:$\bar{\lambda}_{1}\geq 0$, there is a positive constant $C_{1}$ such that for $u\geq s$
 $$\parallel\frac{\partial^{i}g_{u,s}}{\partial x_{1}^{i_{1}}\cdots x_{n-1}^{i_{n-1}}}(x_{1},\cdots,x_{n-1})\parallel\leq C_{1}e^{((l+1)\delta+i(\bar{\lambda}_{1}+\varepsilon+\delta))(u-s)},$$
 $i=i_{1}+\cdots+i_{n-1}=1,\cdots,m+1-l.$

 for any $t_{2}>t_{1}>s, i=1,\cdots,m+1-l,$ we get
\begin{equation}
\begin{split}
&\parallel\int_{t_{1}}^{t_{2}}\begin{pmatrix}
\Lambda_1 &         &        &\\
         &\Lambda_2 &        &\\
         &         &\ddots  &\\
         &         &        &\Lambda_{(n-1)}
\end{pmatrix}\cdot\Lambda_{n}(u,s)^{l}\cdot\frac{\partial^{i}g_{u,s}}{\partial x_{1}^{i_{1}}\cdots x_{n-1}^{i_{n-1}}}(x_{1},\cdots,x_{n-1})du\parallel\\
&\leq \hat{C}_{1}\frac{e^{(t_{2}-s)\alpha_{1}}-e^{(t_{1}-s)\alpha_{1}}}{\alpha_{1}}
\end{split}
\end{equation}

where
$$
\hat{C}_{1}=C_{1}K(\varepsilon)^{l+1}, \alpha_{1}=l\bar{\lambda}_{n}-\underline{\lambda}_{n-1}+i\bar{\lambda}_{1}+(l+i+1)(\varepsilon+\delta)
$$
Using spectral gap condition, we get the function $h_{s}$ is $C^{m+1-l}$.

Case2:$\bar{\lambda}_{1}<0$, there is a positive constant $C_{2}$ such that for $u\geq s$
$$ \parallel\frac{\partial^{i}g_{u,s}}{\partial x_{1}^{i_{1}}\cdots x_{n-1}^{i_{n-1}}}(x_{1},\cdots,x_{n-1})\parallel\leq C_{2}e^{((l+1)\delta+(\bar{\lambda}_{1}+\varepsilon+\delta))(u-s)}$$
$ i=i_{1}+\cdots+i_{n-1}=1,\cdots,m+1-l$

for any $t_{2}>t_{1}>s, i=1,\cdots,m+1-l,$ we get
\begin{equation}
\begin{split}
&\parallel\int_{t_{1}}^{t_{2}}\begin{pmatrix}
\Lambda_1 &         &        &\\
         &\Lambda_2 &        &\\
         &         &\ddots  &\\
         &         &        &\Lambda_{(n-1)}
\end{pmatrix}\cdot\Lambda_{n}(u,s)^{l}\cdot \\
&\frac{\partial^{i}g_{u,s}}{\partial x_{1}^{i_{1}}\cdots x_{n-1}^{i_{n-1}}}(x_{1},\cdots,x_{n-1})du\parallel\\
&\leq \hat{C}_{1}\frac{e^{(t_{2}-s)\alpha_{2}}-e^{(t_{1}-s)\alpha_{2}}}{\alpha_{2}}.
\end{split}
\end{equation}

where
$$
\hat{C}_{2}=C_{2}K(\varepsilon)^{l+1}, \alpha_{2}=l\bar{\lambda}_{n}-\underline{\lambda}_{n-1}+\bar{\lambda}_{1}+(l+2)(\varepsilon+\delta).
$$

Using spectral gap condition, we get the function $h_{s}$ is $C^{m+1-l}$.\\

Step 4:$\frac{\partial h_{s}}{\partial s} ~~is~~ C^{m-l}$.\\
Analog to the proof in Step 3, we get that $h_{s}$ is $C^{m-l}$.\\

Step 5:We define the following map :
$\tilde{\varphi}:\mathbb{R}\times \mathbb{R}\times \mathbb{R}^{n}\rightarrow\mathbb{R}^{n}$ by

$$\tilde{\varphi}(t,s,x)=H_{t}^{-1}\varphi(t,s)H_{s}(x),$$
where
$$H_{t}(x)=\begin{pmatrix}
x_{1}+h_{t,1}(x_{1},\cdots,x_{n-1})x_{n}^{l}\\
x_{2}+h_{t,2}(x_{1},\cdots,x_{n-1})x_{n}^{l}\\
\vdots\\
x_{n-1}+h_{t,n-1}(x_{1},\cdots,x_{n-1})x_{n}^{l}\\
x_{n}
\end{pmatrix}.
$$

By the Implicit Function Theorem ,
$$
H_{t}^{-1}(x)=\begin{pmatrix}
x_{1}+\alpha_{t,1}(x_{1},\cdots,x_{n-1})x_{n}^{l}\\
x_{2}+\alpha_{t,2}(x_{1},\cdots,x_{n-1})x_{n}^{l}\\
\vdots\\
x_{n-1}+\alpha_{t,n-1}(x_{1},\cdots,x_{n-1})x_{n}^{l}\\
x_{n}
\end{pmatrix}$$
where $\alpha_{t} ~~\text{is}~~ C^{m+1-l}, \frac{d\alpha_{t}}{dt} ~~\text{is}~~ C^{m-l}$. We show that the nonautonomous differential equation corresponding to $\tilde{\varphi}$ is of the form of Eq.(11) and this completes the proof of this proposition.\\
Define
\begin{equation}
  \tilde{f}(t,x):=\frac{\partial \tilde{\varphi}}{\partial t}(t,s,x).
\end{equation}
By $\tilde{\varphi}(t,s\tilde{\varphi}(s,\tau,x))=\tilde{\varphi}(t,\tau,x)$, we get
\begin{equation}
\begin{split}
 \frac{\partial \tilde{\varphi}}{\partial t}(t,s,x)&=\lim\limits_{\Delta\rightarrow 0}\frac{\tilde{\varphi}(t+\Delta,s,x)-\tilde{\varphi}(t,s,x)}{\Delta}\\
 &=\lim\limits_{\Delta\rightarrow 0}\frac{\tilde{\varphi}(t+\Delta,t,\tilde{\varphi}(t,s,x))-\tilde{\varphi}(t,s,x)}{\Delta}\\
 &=\tilde{f}(t,\tilde{\varphi}(t,s,x)),
\end{split}
\end{equation}

which implies $\tilde{\varphi}$ is a solution of $\dot{x}=\tilde{f}(t,x).$\\
Morever
$$
\tilde{\varphi}(t,s,x)=$$
$$\begin{pmatrix}
\varphi_{1}(t,s,H_{s}(x))+\alpha_{t,1}(\varphi_{1}(t,s,H_{s}(x)),\cdots,\varphi_{n-1}(t,s,H_{s}(x)))\varphi_{n}(t,s,H_{s}(x))^{l}\\
\varphi_{2}(t,s,H_{s}(x))+\alpha_{t,2}(\varphi_{1}(t,s,H_{s}(x)),\cdots,\varphi_{n-1}(t,s,H_{s}(x)))\varphi_{n}(t,s,H_{s}(x))^{l}\\
\vdots\\
\varphi_{n-1}(t,s,H_{s}(x))+\alpha_{t,n-1}(\varphi_{1}(t,s,H_{s}(x)),\cdots,\varphi_{n-1}(t,s,H_{s}(x)))\varphi_{n}(t,s,H_{s}(x))^{l}\\
\varphi_{n}(t,s,H_{s}(x)),
\end{pmatrix}
$$
which implies that $\tilde{f}$ is in the class $C(m,n)$.\\
Next
$$
h_{t}(x_{1},\cdots,x_{n-1})=\begin{pmatrix}
\lim\limits_{\tau\rightarrow\infty}\lim\limits_{x_{n}\rightarrow 0}\frac{\varphi_{1}(t,\tau,\phi(\tau,t,x))-x_{1}}{x_{n}^{l}}\\
\lim\limits_{\tau\rightarrow\infty}\lim\limits_{x_{n}\rightarrow 0}\frac{\varphi_{2}(t,\tau,\phi(\tau,t,x))-x_{2}}{x_{n}^{l}}\\
\cdots\\
\lim\limits_{\tau\rightarrow\infty}\lim\limits_{x_{n}\rightarrow 0}\frac{\varphi_{n-1}(t,\tau,\phi(\tau,t,x))-x_{n-1}}{x_{n}^{l}}.\\
\end{pmatrix}
$$
Consequently, we have
\begin{equation}
\begin{split}
&h_{t}(\phi_{1}(t,s,x_{1},\cdots,x_{n-1}),\cdots,\phi_{1}(t,s,x_{1},\cdots,x_{n-1}))\\
&=\begin{pmatrix}
\lim\limits_{\tau\rightarrow\infty}\lim\limits_{x_{n}\rightarrow 0}\frac{\varphi_{1}(t,\tau,\phi(\tau,t,x))-\phi_{1}(t,s,x_{1},\cdots,x_{n-1})}{\phi_{n}(t,s,x)^{l}}\\
\lim\limits_{\tau\rightarrow\infty}\lim\limits_{x_{n}\rightarrow 0}\frac{\varphi_{2}(t,\tau,\phi(\tau,t,x))-\phi_{2}(t,s,x_{1},\cdots,x_{n-1})}{\phi_{n}(t,s,x)^{l}}\\
\cdots\\
\lim\limits_{\tau\rightarrow\infty}\lim\limits_{x_{n}\rightarrow 0}\frac{\varphi_{n-1}(t,\tau,\phi(\tau,t,x))-\phi_{n-1}(t,s,x_{1},\cdots,x_{n-1})}{\phi_{n}(t,s,x)^{l}}\\
\end{pmatrix}\\
&=\begin{pmatrix}
\lim\limits_{\tau\rightarrow\infty}\lim\limits_{x_{n}\rightarrow 0}\frac{\varphi_{1}(t,s)\varphi(s,\tau)\phi(\tau,s)x-\phi_{1}(t,s,x_{1},\cdots,x_{n-1})}{\phi_{n}(t,s,x)^{l}}\\
\lim\limits_{\tau\rightarrow\infty}\lim\limits_{x_{n}\rightarrow 0}\frac{\varphi_{2}(t,s)\varphi(s,\tau)\phi(\tau,s)x-\phi_{2}(t,s,x_{1},\cdots,x_{n-1})}{\phi_{n}(t,s,x)^{l}}\\
\cdots\\
\lim\limits_{\tau\rightarrow\infty}\lim\limits_{x_{n}\rightarrow 0}\frac{\varphi_{n-1}(t,s)\varphi(s,\tau)\phi(\tau,s)x-\phi_{n-1}(t,s,x_{1},\cdots,x_{n-1})}{\Phi_{n}(t,s,x)^{l}}\\
\end{pmatrix}\\
&=\begin{pmatrix}
\lim\limits_{\tau\rightarrow\infty}\lim\limits_{x_{n}\rightarrow 0} \frac{\varphi_{1}(t,s,x^{'})-\phi_{1}(t,s,x_{1},\cdots,x_{n-1})}{\phi_{n}(t,s,x)^{l}}\\
\lim\limits_{\tau\rightarrow\infty}\lim\limits_{x_{n}\rightarrow 0} \frac{\varphi_{2}(t,s,x^{'})-\phi_{2}(t,s,x_{1},\cdots,x_{n-1})}{\phi_{n}(t,s,x)^{l}}\\
\cdots\\
\lim\limits_{\tau\rightarrow\infty}\lim\limits_{x_{n}\rightarrow 0} \frac{\varphi_{n-1}(t,s,x^{'})-\phi_{n-1}(t,s,x_{1},\cdots,x_{n-1})}{\phi_{n}(t,s,x)^{l}}\\
\end{pmatrix}\\
&=\begin{pmatrix}
\lim\limits_{x_{n}\rightarrow 0} \frac{\varphi_{1}(t,s,x^{''})-\phi_{1}(t,s,x_{1},\cdots,x_{n-1})}{\phi_{n}(t,s,x)^{l}}\\
\lim\limits_{x_{n}\rightarrow 0} \frac{\varphi_{2}(t,s,x^{''})-\phi_{2}(t,s,x_{1},\cdots,x_{n-1})}{\phi_{n}(t,s,x)^{l}}\\
\cdots\\
\lim\limits_{x_{n}\rightarrow 0} \frac{\varphi_{n-1}(t,s,x^{''})-\phi_{n-1}(t,s,x_{1},\cdots,x_{n-1})}{\phi_{n}(t,s,x)^{l}}\\
\end{pmatrix},
\end{split}
\end{equation}
where
$$x^{'}=x_{1}+h_{s,\tau,1}(x_{1},\cdots,x_{n-1})x_{n}^{l},\cdots,x_{n-1}+h_{s,\tau,n-1}(x_{1},\cdots,x_{n-1})x_{n}^{l},x_{n}$$
$$x^{''}=x_{1}+h_{s,1}(x_{1},\cdots,x_{n-1})x_{n}^{l},\cdots,x_{n-1}+h_{s,n-1}(x_{1},\cdots,x_{n-1})x_{n}^{l},x_{n}.$$\\
which implies that
\begin{equation}
\begin{split}
&\phi_{j}(t,s,x_{1},\cdots,x_{n-1})\\
&+h_{t}(\phi_{1}(t,s,x_{1},\cdots,x_{n-1}),\cdots,\phi_{n-1}(t,s,x_{1},\cdots,x_{n-1}))\phi_{n}(t,s,x)^{l}\\
&=\varphi_{j}(t,s,x_{1}+h_{s,1}(x_{1},\cdots,x_{n-1})x_{n}^{l},\cdots,x_{n-1}+h_{s,n-1}(x_{1},\cdots,x_{n}^{l}),x_{n})\\
&+O(x_{n}^{l}).
\end{split}
\end{equation}
Therefore,by definition of $\tilde{\varphi}$ we get
$$\tilde{\varphi}_{j}(t,s,x_{1},\cdots,x_{n})=\phi_{j}(t,s,x_{1},\cdots,x_{n-1})+O(x_{n}^{l}),$$
which completes the proof of Proposition \ref{p1}.
\end{proof}

\begin{equation}
\label{eq3}
\dot{x}=B(t)x+P(t,x)+\begin{pmatrix}
b_{1}(t,x_{1},\cdots,x_{n-1})x_{n}^{l+1}\\
         \cdots\\
b_{n-1}(t,x_{1},\cdots,x_{n-1})x_{n}^{l+l}\\
b_{n}(t,x_{1},\cdots,x_{n-1})x_{n}^{l+l}
\end{pmatrix}+r_{5}(t,x)
\end{equation}
where $b_{i}(t,x_{1},\cdots,x_{n-1})$, $1\leq i\leq n$ are $C^{m-l}$ Carath\'{e}odory functions and $r_{5}(t,x)$ belongs to $C(m+1,n)$ and $r_{3}(t,x)=O(x_{n}^{l+2})$.
\begin{equation}
\label{eq4}
\dot{x}=B(t)x+P(t,x)+r_{6}(t,x)
\end{equation}
where $r_{6}(t,x)$ belongs to $C(m,n)$, $r_{6j}(t,x)=O(x_{n}^{l+1})$, $1\leq j\leq n-1$ and $r_{6n}(t,x)=O(x_{n}^{l+2})$.\\
\textbf{Condition 2}\\
a)$\lambda_{n}<\cdots<\lambda_{1}$, $\lambda_{n}<0$;\\
b)$\bar{\lambda}_{n}+i\bar{\lambda}_{1}+(l+i)(\varepsilon+\delta)<0$ as $\bar{\lambda}_{1}\geq0$;\\
c)$ l\bar{\lambda}_{n}+\bar{\lambda}_{1}+(l+1)(\varepsilon+\delta)<0$ as $ \bar{\lambda}_{1}<0$.\\
where $1\leq l\leq n-1$ and $1\leq i\leq n$, $\varepsilon$ and $\delta$ satisfied Theorem \ref{th3}.
\begin{prop}
\label{2}
If \textbf{Condition 2} hold, then system (\ref{eq3}) is $C^{m-l}$ equivalent to the system (\ref{eq4}).
\end{prop}

\begin{proof}
Let $\varphi(t,s,x)$ denote the solution of (30) with the initial condition $x(s)=(x_{1},\cdots,x_{n})$. Recall that $\Lambda(t,s)x$ denotes the solution of the following system
\[
\dot{x}=B(t)x,
\] with the initial condition $x(s)=x$,
$\phi(t,s,x)$ denotes the solution of the following system
\[
\dot{x}=B(t)x+P(t,x).
\] with the initial condition $x(s)=x$.\\
In order to construct a smooth transformation between system (30) and (31), we define a scalar function $h_{s,t}:\mathbb{R}^{n-1}\rightarrow\mathbb{R}$ for $s,t\in \mathbb{R}$, by
$$h_{s,t}(x_{1},\cdots,x_{n-1}):=\lim\limits_{x_{n}\rightarrow 0}\frac{\varphi_{n}(s,t)\Phi(t.s)(x)-x_{n}}{x_{n}^{l+1}}~~for~~x_{1},\cdots,x_{n-1}\in \mathbb{R}.$$
In the following steps, we investigate some properties of the function $h_{s,t}$ and explain how to construct the desired transformation by using the function $h_{s,t}$.\\
Step 1:As the step 1 of Proposition 5.1, the following equation holds
\begin{equation}
h_{s,t}(x_{1},\cdots,x_{n-1})=-\int_{s}^{t}\Lambda_{n}(u,s)^{l}g_{u,s}(x_{1},\cdots,x_{n-1})du,
\end{equation}
where
\begin{equation}
g_{u,s}(x_{1},\cdots,x_{n-1})=b_{3}(u,\phi_{1},\cdots,\phi_{n-1})e^{l\int_{s}^{u}p_{n}(v,\phi_{1},\cdots,\phi_{n-1})dv}.
\end{equation}
Step 2: In this step, \\
$h_{s}:\mathbb{R}^{n-1}\rightarrow\mathbb{R}$ and $h_{s}(x_{1},\cdots,x_{n-1})=\lim\limits_{t\rightarrow\infty}h_{s,t}(x_{1},\cdots,x_{n-1})$,

let $t_{2}>t_{1}>s$, then
$$
\parallel h_{s,t_{1}}(x_{1},\cdots,x_{n-1})-h_{s,t_{2}}(x_{1},\cdots,x_{n-1})\parallel$$
$$\leq M\cdot K(\varepsilon)^{l}\int_{t_{1}}^{t_{2}}e^{l(\bar{\lambda_{n}}+\varepsilon+\delta)(u-s)}du.
$$
Since $\bar{\lambda_{n}}+\varepsilon+\delta<0$, it follows that the limit $\lim\limits_{t\rightarrow\infty}h_{s,t}(x_{1},\cdots,x_{n-1})$ exists uniformly in $x_{1},\cdots,x_{n-1}$.\\
Step 3:$h_{s}$ is $C^{m-l}$, that is to say, the following improper integral is uniform convergence
\begin{equation}
\int_{s}^{\infty}\Lambda_{n}(u,s)^{l}\cdot \frac{\partial^{i}g_{u,s}}{\partial x_{1}^{i_{1}}\cdots x_{n-1}^{i_{n-1}}}(x_{1},\cdots,x_{n-1})du,
\end{equation}
for $i=i_{1}+\cdots+i_{n-1}=1,\cdots,m-l$
For this purpose, we consider the following two cases $\bar{\lambda}_{1}\geq 0$ and $\bar{\lambda_{1}}< 0$.\\
Case 1:$\bar{\lambda}_{1}\geq 0$. there exists a positive constant $C_{1}$ such that for $u\geq s,  i=i_{1}+\cdots+i_{n-1}=1,\cdots,m-l$
$$\parallel\frac{\partial^{i}g_{u,s}}{\partial x_{1}^{i_{1}}\cdots x_{n-1}^{i_{n-1}}}(x_{1},\cdots,x_{n-1})\parallel\leq C_{1}e^{(l\delta+i(\bar{\lambda}_{1}+\varepsilon+\delta))(u-s)}.
$$
Consequently, for any $t_{2}>t_{1}>s, i=1,\cdots,m-l$, we get
$$
\parallel\int_{t_{1}}^{t_{2}}\Lambda_{n}(u,s)^{l}\cdot \frac{\partial^{i}g_{u,s}}{\partial x_{1}^{i_{1}}\cdots x_{n-1}^{i_{n-1}}}(x_{1},\cdots,x_{n-1})du\parallel$$
$$\leq \hat{C}_{1}\frac{e^{(t_{2}-s)\alpha_{1}}-e^{(t_{1}-s)\alpha_{1}}}{\alpha_{1}},
$$
where
$$\hat{C}_{1}=C_{1}K(\varepsilon)^{l}, $$ $$\alpha_{1}=\bar{\lambda}_{n}+i\bar{\lambda}_{1}+(l+i)(\varepsilon+\delta),$$
which proves the uniform convergence in $x_{1},\cdots,x_{n-1}$ in (34). Therefore, the function $h_{s}$ is $C_{m-l}$.\\
Case 2:$\bar{\lambda}_{1}< 0$. there exists a positive constant $C_{2}$ such that for $u\geq s,  i=i_{1}+\cdots+i_{n-1}=1,\cdots,m-l$
$$
\parallel \frac{\partial^{i}g_{u,s}}{\partial x_{1}^{i_{1}}\cdots x_{n-1}^{i_{n-1}}}(x_{1},\cdots,x_{n-1})\parallel\leq \hat{C}_{2}e^{(l\delta+(\bar{\lambda}_{1}+\varepsilon+\delta))(u-s)}.
$$
Consequently, for any $t_{2}>t_{1}>s, i=1,\cdots,m-l$, we get
$$
\parallel\int_{t_{1}}^{t_{2}}\Lambda_{n}(u,s)^{l}\cdot\frac{\partial^{i}g_{u,s}}{\partial x_{1}^{i_{1}}\cdots x_{n-1}^{i_{n-1}}}(x_{1},\cdots,x_{n-1})du\parallel$$
$$\leq \hat{C}_{2}\frac{e^{(t_{2}-s)\alpha_{2}}-e^{(t_{1}-s)\alpha_{2}}}{\alpha_{2}},
$$
where
$$\hat{C}_{2}=C_{2}K(\varepsilon)^{l+1}, $$ $$\alpha_{2}=l\bar{\lambda}_{n}+\bar{\lambda}_{1}+(l+1)(\varepsilon+\delta).$$
Step 4:$\frac{\partial h_{s}}{\partial s} ~~is~~ C^{m-l}$,
 Analog to the proof in  Step 3.\\
Step 5:Using the function $h_{s}$, we define the following function
$$\tilde{\varphi}:\mathbb{R}\times \mathbb{R}\times \mathbb{R}^{n}\rightarrow\mathbb{R}^{n}$$
$$\tilde{\varphi}(t,s,x)=H_{t}^{-1}\varphi(t,s)H_{s}(x)$$
where $H_{t}(x)=\begin{pmatrix}
x_{1}\\
x_{2}\\
\vdots\\
x_{n-1}\\
x_{n}+h_{t}(x_{1},\cdots,x_{n-1})x_{n}^{l+1}
\end{pmatrix}$,
Similar to Step 5 in the proof of Proposition \ref{p1}, we can proof $\tilde{\varphi}$ is a solution of the equation of the form of (31). This completes the proof of the proposition \ref{2}.
\end{proof}
Applying Proposition \ref{p1} and Proposition \ref{2}, we can prove that (8) is $C^{k}$ equivalent to the following system
\begin{equation}
\label{eq5}
 \dot{x}=B(t)x+P(t,x)+r_{7}(t,x),
\end{equation}
where $r_{7}=O(x_{n}^{n})$.
\begin{equation}
\label{eq6}
 \dot{x}=B(t)x+P(t,x).
\end{equation}

\subsection{Estimate of the remaining term}
\textbf{Condition 3}\\
a)$\lambda_{n}<\cdots<\lambda_{1}$, $\lambda_{n}<0$;\\
b)$n\bar{\lambda}_{n}-\underline{\lambda}_{n}+(n+1)(\varepsilon+\delta)<0$.\\
where $\varepsilon$ and $\delta$ satisfied Theorem \ref{th3}.
\begin{prop}
\label{p3}
If \textbf{Condition 3} hold, then system (\ref{eq5}) is $C^{k}$ equivalent to the system (\ref{eq6}).
\end{prop}
\begin{proof}
In order to construct a smooth transformation between system (\ref{eq5}) and system (\ref{eq6}), we define the following function $H_{s}(t,\cdot):\mathbb{R}^{n}\rightarrow\mathbb{R}^{n}, s,t\in \mathbb{R}$,by
\begin{equation}
  H_{s}(t,x)=\varphi(s,t,\phi(t,s,x))\quad \text{for}\quad x\in \mathbb{R}^{n}.
\end{equation}
Step 1: we will show that
\begin{equation}
   H_{s}(t,x)=x-\int_{s}^{t}\frac{\partial \varphi}{\partial x}(s,u,\phi(u,s,x))r_{7}(u,\phi(u,s,x))du,
\end{equation}
taking derivative with respect to the t-variable in (38) leads to
\begin{equation}
  \frac{\partial H_{s}}{\partial t}(t,x)=\frac{\partial \varphi}{\partial t}(s,t,\phi(t,s,x))+\frac{\partial \varphi}{\partial x}(s,t,\phi(t,s,x))\frac{d}{dt}\phi(t,s,x).
\end{equation}

By definition of $\phi$, we get that
\begin{equation}
\frac{d}{dt}\phi(t,s,x)=B(t)\phi(t,s,x)+\begin{pmatrix}
p_{1}(t,\phi^{'}(t,s,x))\\
         \cdots\\
p_{n-1}(t,\phi^{'}(t,s,x))\\
p_{n}(t,\phi^{'}(t,s,x))\phi_{n}(t,s,x)
\end{pmatrix},
\end{equation}
where $\phi^{'}(t,s,x)=\phi_{1}(t,s,x),\cdots,\phi_{n-1}(t,s,x).$\\
Therefore,
\begin{equation}
\begin{split}
   \frac{\partial \varphi}{\partial t}(s,t,x)&=\lim\limits_{\Delta\rightarrow 0}\frac{\varphi(s,t+\Delta,z)-\varphi(s,t,z)}{\Delta}\\
   &=\lim\limits_{\Delta\rightarrow 0}\frac{\varphi(s,t,\varphi(t,t+\Delta,z))-\varphi(s,t,z)}{\Delta}\\
   &=\frac{\partial \varphi}{\partial x}(s,t,z)\lim\limits_{\Delta\rightarrow 0}\frac{\varphi(s,t+\Delta,z)-z}{\Delta}.
\end{split}
\end{equation}
Since
$$\varphi(s,t+\Delta,z)=z-\Delta\left(B(t)z+\begin{pmatrix}
p_{1}(t,z^{'})\\
         \cdots\\
p_{n-1}(t,z^{'})\\
p_{n}(t,z^{'})z_{n}
\end{pmatrix}+r_{7}(t,z)
\right)+O(\Delta^{n}),
$$
it follows that
$$\frac{\partial \varphi}{\partial t}(t,t,x)=-B(t)z-\begin{pmatrix}
p_{1}(t,z^{'})\\
         \cdots\\
p_{n-1}(t,z^{'})\\
p_{n}(t,z^{'})z_{n}
\end{pmatrix}-r_{7}(t,z),$$
where $z^{'}=z_{1},\cdots,z_{n-1}.$\\
Replacing $z$ by $\phi(t,s,x)$, we obtain that
\begin{equation}
\begin{split}
&\frac{\partial \varphi}{\partial t}(t,t,\phi(t,s,x))=-B(t)\phi(t,s,x)\\
&-\begin{pmatrix}
p_{1}(t,\phi^{'}(t,s,x))\\
         \cdots\\
p_{n-1}(t,\phi^{'}(t,s,x))\\
p_{n}(t,\phi^{'}(t,s,x))\phi_{n}(t,s,x)
\end{pmatrix}-r_{7}(t,\phi(t,s,x)),
\end{split}
\end{equation}
which together with (41) gives that
\begin{equation}
\begin{split}
 & \frac{\partial \varphi}{\partial t}(t,t,\phi(t,s,x))=-\frac{\partial \varphi}{\partial x}(s,t,\phi(t,s,x))[B(t)\phi(t,s,x)\\
&+\begin{pmatrix}
p_{1}(t,\phi^{'}(t,s,x))\\
         \cdots\\
p_{n-1}(t,\phi^{'}(t,s,x))\\
p_{n}(t,\phi^{'}(t,s,x))\phi_{n}(t,s,x)
\end{pmatrix}+r_{7}(t,\phi(t,s,x))].
\end{split}
\end{equation}

Combining the above equality and (39) and (40), we obtain that
\begin{equation}
  \frac{\partial H_{s}}{\partial t}(t,x)=-\frac{\partial \varphi}{\partial x}(s,t,\phi(t,s,x))r(t,\phi(t,s,x)),
\end{equation}
which together with the fact that $H_{s}(s,x)=x$ proves (38).\\
Step 2: In this step, we show that the map
\begin{equation}
  H_{s}(x)=\lim\limits_{t\rightarrow\infty}H_{s}(t,x)=x-\int_{s}^{\infty}\frac{\partial \varphi}{\partial x}(s,u,\phi(u,s,x))r(u,\phi(u,s,x))du
\end{equation}
is well-defined and fulfills the following properties:

\begin{enumerate}
\item $\lim\limits_{x\rightarrow0}H_{t}(x)=0$ uniformly in t,
\item $H$ maps a solution of (36) to a solution of (35).
\end{enumerate}

For this purpose, let $t_{1}>t_{2}\geq s$ be arbitrary. we obtain that for some positive constant $L$
\begin{equation}
\begin{split}
  &\parallel\int_{t_{2}}^{t_{1}}\frac{\partial \varphi}{\partial x}(s,u,\phi(u,s,x))r_{7}(u,\phi(u,s,x))du\parallel\\
  &\leq L\int_{t_{2}}^{t_{1}}e^{(u-s)(\varepsilon+\delta-\underline{\lambda}_{n})}\parallel r_{7}(u,\phi(u,s,x))\parallel du\\
  &\leq CLK(\varepsilon)^{n}|x_{n}|^{n}\int_{t_{2}}^{t_{1}}e^{(u-s)(n\bar{\lambda}_{n}-\underline{\lambda}_{n}+(n+1)(\varepsilon+\delta))}du
\end{split}
\end{equation}
For $n\bar{\lambda}_{n}-\underline{\lambda}_{n}+(n+1)(\varepsilon+\delta)<0$, the following improper integral
$$\int_{s}^{\infty}\frac{\partial \varphi}{\partial x}(s,u,\phi(u,s,x))r_{7}(u,\phi(u,s,x))du,$$
exists. \\
Furthermore, we have
\begin{equation}
\begin{split}
  \parallel H_{s}-x \parallel &\leq CLK(\varepsilon)^{n}|x_{n}|^{n}\int_{s}^{\infty}e^{(u-s)(n\bar{\lambda}_{n}-\underline{\lambda}_{n}+(n+1)(\varepsilon+\delta))}du\\
  &=\frac{CLK(\varepsilon)^{n}}{n\bar{\lambda}_{n}-\underline{\lambda}_{n}+(n+1)(\varepsilon+\delta)}|x_{n}|^{n},
\end{split}
\end{equation}
which proves (i). By definition of $H$, we get that for $t,s\in \mathbb{R}$ and $x\in \mathbb{R}^{n}$
\begin{equation}
\begin{split}
H_{t}(\phi(t,s,x))&=\lim\limits_{u\rightarrow \infty}\varphi(t,u,\phi(u.t)\phi(t,s,x))\\
&=\varphi(t,u,\lim\limits_{u\rightarrow \infty}\varphi(s,u,\phi(u,s,x)))\\
&=\varphi(t,s,H_{s}(x)),
\end{split}
\end{equation}
which proves that $H$ maps a solution of (36) to a solution of (35).

Step 3:$H_{s}$ is $C^{k}$.\\
Anology to Step 3 of proposition \ref{p1}, we can prove step 3 of proposition \ref{p3}. According to step 1, step 2, step 3, proposition \ref{p3} is proved.
\end{proof}
\textbf{Condition 4} \\
a) $\lambda_{n}<\cdots<\lambda_{1}$, $\lambda_{n}<0$;\\
b) $k\lambda_{n}+m\lambda_{1}-m\lambda_{n}<0 $ as $ \bar{\lambda}_{1}\geq0$ and $m\geq n$ be arbitrary;\\
c) $k\lambda_{n}+\lambda_{1}-m\lambda_{n}<0~and~k\lambda_{n}-\lambda_{n}<0$ as $ \bar{\lambda}_{1}<0$ and $m\geq n$ be arbitrary.

\begin{thm}
\label{th}
If \textbf{Condition 4}  hold,
then system (4) is $C^{k}$ equivalent to the following system
\begin{equation}
\dot{x}=B(t)x+\begin{pmatrix}
p_{1}(t,x_{1})\\
p_{2}(t,x_{1})x_{2}\\
         \cdots\\
p_{n}(t,x_{1})x_{n}
\end{pmatrix}.
\end{equation}
\end{thm}
\begin{proof}
If \textbf{Condition 4} hold, then \textbf{Condition 1} and \textbf{Condition 2} hold. According to theorem \ref{th2}, system(4) is $C^{k}$ equivalent to system (5); According to theorem \ref{th3}, system(5) is $C^{k}$ equivalent to system (7); Using the result of section 3, system (7) is $C^{k}$ equivalent to system (8). From Proposition \ref{p1}, Proposition \ref{2} and Proposition \ref{p3}, we can see that system (8) is $C^{k}$ equivalent to (36). Repeating using Proposition \ref{p1}, Proposition \ref{2} and Proposition \ref{p3}, we can get system (36) is $C^{k}$ equivalent to system (49). Then Theorem \ref{th} is proved.
\end{proof}

\section{Example}
 Consider the system
\begin{equation}
\begin{split}
\dot{x}=&\begin{pmatrix}
\alpha+\sigma\xi(t)&-\alpha-\sigma\xi(t)  & 0    \\
\alpha+\sigma\xi(t)&-1-\alpha-\sigma\xi(t)& 0 \\
       0           & 0                    &-1-\alpha-\sigma\xi(t)
\end{pmatrix}x-\\
&\begin{pmatrix}
x_{1}^{4}-2x_{1}x_{2}x_{3}^{2}+x_{2}x_{3}^{2}\\
x_{1}^{3}x_{2}-2x_{1}^{2}x_{2}^{2}+x_{2}x_{3}^{3}\\
x_{1}^{3}x_{3}-2x_{1}x_{2}x_{3}^{2}+x_{2}x_{3}^{3}\\
\end{pmatrix}
\end{split}
\end{equation}
By using the result\cite{article10} of Palmer, when $\alpha, \sigma$ are sufficiently small, the dichotomy spectral of the following system
$$\dot{x}=\begin{pmatrix}
\alpha+\sigma\xi(t)&-\alpha-\sigma\xi(t)  & 0    \\
\alpha+\sigma\xi(t)&-1-\alpha-\sigma\xi(t)& 0 \\
       0           & 0                    &-1-\alpha-\sigma\xi(t)
\end{pmatrix}x$$ consist of three disjoint intervals. Meanwhile, there exist $k,N$ such that the gap conditions in theorem 3.1 are satisfied. So (49) is $C^{k}$ equivalent to a system of the following form
$$\dot{x}=\Lambda(t)x+\begin{pmatrix}
p_{1}(t,x_{1},0)\\
p_{2}(t,x_{1},0)x_{2}\\
p_{3}(t,x_{1},0)x_{3}\\
\end{pmatrix},
$$
where $p_{1},p_{2},p_{3}$ are $C^{k}$ Carath\'{e}odory functions and
$$p_{1}(t,0,0)=p_{2}(t,0,0)=p_{3}(t,0,0)=\frac{\partial p_{1}}{\partial x_{1}}(t,0,0)=0.$$
As for the form of functions $p_{1},p_{2},p_{3}$, we can compute it like the computation of the invariant manifold \cite{article11}.

\end{document}